\documentclass[reqno,a4paper,12pt]{amsart}
\usepackage{a4wide}

\usepackage{amsmath,amssymb,amscd,xcolor,amsthm,graphicx,enumerate,float}

\usepackage{hyperref}
\hypersetup{breaklinks=true}

\numberwithin{equation}{section}
\theoremstyle{plain}

\makeatletter
\newtheorem*{rep@theorem}{\rep@title}
\newcommand{\newreptheorem}[2]{%
\newenvironment{rep#1}[1]{%
 \def\rep@title{#2 \ref{##1}}%
 \begin{rep@theorem}}%
 {\end{rep@theorem}}}
\makeatother

\newtheorem{theorem}[equation]{Theorem}
\newreptheorem{theorem}{Theorem}

\newtheorem{proposition}[equation]{Proposition}
\newtheorem{lemma}[equation]{Lemma}

\newtheorem{claim}[equation]{Claim}

\theoremstyle{remark}
\newtheorem{remark}[equation]{Remark}

\theoremstyle{definition}
\newtheorem{definition}[equation]{Definition}

\newtheorem*{question*}{Question}

\newcommand{\A}{{\mathcal A}}

\newcommand{\R}{\mathbb R}

\newcommand{\eps}{\varepsilon}

\providecommand{\abs}[1]{\lvert #1\rvert}

\begin{document}

\title[The moduli space of two-convex embedded tori]{The moduli space of two-convex embedded tori}
\author{Reto Buzano \and Robert Haslhofer \and Or Hershkovits}

\date{\today \thanks{R.B. has been supported by EPSRC grant EP/M011224/1. R.H. has been supported by NSERC grant RGPIN-2016-04331, NSF grant DMS-1406394 and a Connaught New Researcher Award. O.H. has been supported by an AMS-Simons travel grant.}}

\begin{abstract}
In this short article we investigate the topology of the moduli space of two-convex embedded tori $S^{n-1}\times S^1\subset \mathbb{R}^{n+1}$. We prove that for $n\geq 3$ this moduli space is path-connected, and that for $n=2$ the connected components of the moduli space are in bijective correspondence with the knot classes associated to the embeddings.
Our proof uses a variant of mean curvature flow with surgery developed in our earlier article \cite{BHH1} where neck regions are deformed to tiny strings instead of being cut out completely, an approach which preserves the global topology, embeddedness, as well as two-convexity. 
\end{abstract}

\maketitle


\section{Introduction}
The goal of this short article is to extend the results from our previous paper \cite{BHH1} to the case of two-convex embedded tori. In our previous paper we considered the moduli space of two-convex embedded spheres, i.e. the space
\begin{equation}
\mathcal{M}^{2-\textrm{conv}}(S^{n})=\textrm{Emb}^{2-\textrm{conv}}(S^n,\mathbb{R}^{n+1})/\textrm{Diff}(S^n)
\end{equation}
equipped with the smooth topology, where $\textrm{Emb}^{2-\textrm{conv}}\subset \textrm{Emb}$ denotes the space of smooth embeddings with the property that the sum of the smallest two principal curvatures is positive at every point. We proved that $\mathcal{M}^{2-\textrm{conv}}(S^{n})$ is path-connected in every dimension $n$, and conjectured that $\mathcal{M}^{2-\textrm{conv}}(S^{n})$ is actually contractible in every dimension $n$. This was inspired in part by Hatcher's proof of the Smale conjecture \cite{Hatcher,Smale} and by related work of Marques on the moduli space of metrics with positive scalar curvature \cite{Marques}.

Here, we consider the moduli space of two-convex embedded tori, i.e. the space 
\begin{equation}
\mathcal{M}^{2-\textrm{conv}}(S^{n-1}\times S^1)=\textrm{Emb}^{2-\textrm{conv}}(S^{n-1}\times S^1,\mathbb{R}^{n+1})/\textrm{Diff}(S^{n-1}\times S^1)
\end{equation}
equipped with the smooth topology. Recall that by a result of Huisken-Sinestrari \cite{HuiskenSinestrari} the space $\mathcal{M}^{2-\textrm{conv}}(S^{n-j}\times S^j)$ is empty for $2\leq j \leq n-2$, so without loss of generality we can assume $j=1$ right away. A new interesting feature of the moduli space $\mathcal{M}^{2-\textrm{conv}}(S^{n-1}\times S^1)$ compared to $\mathcal{M}^{2-\textrm{conv}}(S^{n})$, is that it can have non-trivial algebraic topology. In fact, some non-trivial algebraic topology can already be spotted at the level of $\pi_0$, and our main theorem gives a complete classification of the path-components of $\mathcal{M}^{2-\textrm{conv}}(S^{n-1}\times S^1)$.

\begin{theorem}\label{thm_main}
The path-components of the moduli space of two-convex embedded tori are given by
\begin{equation}
\pi_0\left(\mathcal{M}^{2-\textrm{conv}}(S^{n-1}\times S^1)\right) \cong
\begin{cases}
\mathcal{K}\left(\mathcal{M}^{2-\textrm{conv}}(S^{1}\times S^1)\right), & n=2,\\
0, & n\geq 3,
\end{cases}
\end{equation}
where $\mathcal{K}$ denotes the set of knot classes. This means that $\mathcal{M}^{2-\textrm{conv}}(S^{n-1}\times S^1)$ is path-connected for $n\geq 3$, while for $n=2$ we have that two mean-convex embedded tori are in the same path-component of their moduli space if and only if they have the same knot class.
\end{theorem}

We recall that for $n=2$ the notion of two-convexity simply becomes the more basic notion of mean-convexity, i.e. the property that the mean curvature vector points inwards at every point. The knot class of an embedded torus $T\subset \mathbb{R}^3$ is defined as follows: Choose an embedding map $\varphi: S^1\times S^1\to \mathbb{R}^3$ with $\varphi(S^1\times S^1)= T$ such that $\varphi$ can be extended to an embedding $\bar{\varphi}: D^2\times S^1\to \mathbb{R}^3$ of the solid torus bounded by $T$. The knot class of $T$ is then defined as the knot equivalence class of
\begin{equation}
S^1\to \mathbb{R}^3, e^{i\theta}\mapsto \varphi(1,e^{i\theta})\, .
\end{equation}
The knot class of an embedded torus is of course well defined, i.e. independent of the choice of extendible parameterization, and preserved under isotopies (see Section \ref{sec_prelim}).

The result of Theorem \ref{thm_main} is new even in the seemingly classical case $n=2$. While the space of embedded tori in $\mathbb{R}^3$ is a well studied classical object, it seems that none of the known topological methods deforming such tori into one another can actually preserve mean-convexity, i.e. the only method we know to prove Theorem \ref{thm_main} even in dimension $n=2$ is by using our geometric analytic approach based on mean curvature flow.

Let us now outline the main steps of the proof of Theorem \ref{thm_main}, following the framework from  our previous article \cite{BHH1}.

Given a $2$-convex embedded closed hypersurface  $M_0\subset \mathbb{R}^{n+1}$, we consider its mean curvature flow with surgery $\{M_t\}_{t\in [0,\infty)}$ as provided by the existence theorem from \cite[Thm.~1.21]{HK_surgery}. The flow always becomes extinct in finite time $T<\infty$.  Around all but finitely many times the flow is a smooth mean curvature flow, but at some finite non-empty collection of times surgeries and/or discarding will occur. 

By the canonical neighborhood theorem \cite[Thm.~1.22]{HK_surgery} each discarded component is either a convex sphere of controlled geometry, a capped-off chain of $\eps$-necks or an $\eps$-loop. This information is sufficient to construct an explicit $2$-convex isotopy from any discarded component to what we call a \emph{marble graph} (see Section \ref{sec_marble}). Roughly speaking, a marble graph is a family of disjoint spheres, smoothly glued to one another along tubular neighborhoods of some admissible curves.

While surgeries disconnect the hypersurfaces into different connected components, for our topological application we eventually have to connect the pieces again. To this end,  we use the $2$-convex connected sum operation from \cite{BHH1}. In this construction, two $2$-convex hypersurfaces are glued together along tiny tubes around admissible curves (so-called \emph{strings}) connecting the hypersurfaces (see Section \ref{sec_prelim} for details). If the string radius $r_s$ is chosen to be much smaller than the surgery scales used in the mean curvature flow, then these different scales barely interact. As in \cite{BHH1}, we can therefore argue by backwards induction on the surgery times that at each time every connected component is isotopic via $2$-convex embeddings to a marble graph, see Theorem \ref{thm_deform_to_marblegraph}. In fact, this result holds for any embedded $2$-convex closed hypersurface in $\mathbb{R}^{n+1}$ without topological assumption.

If the initial hypersurface $M_0$ has the topology of a torus, it must thus be isotopic via $2$-convex embeddings to a marble circuit -- a marble graph with only one loop (see Section \ref{sec_marble}). Having constructed a $2$-convex isotopy from the original torus to a marble circuit, in Section \ref{sec_circuit} we finally show that such a circuit is isotopic (again via $2$-convex embeddings) to an arbitrarily thin tubular neighborhood of a knot $\gamma$. Finally, it is easy to see that such thin tubular neighborhoods are always $2$-convex isotopic for $n\geq 3$, while for $n=2$ they are mean-convex isotopic if and only if they represent the same knot class.

\section{Preliminaries}\label{sec_prelim}

The goal of this section is to collect and prove various preliminary results.

We start by explaining that the knot class of an embedded torus is well defined.

\begin{definition}[Knot class]\label{knot_class_def}
Two embedded closed curves $\gamma_i:S^1\rightarrow \mathbb{R}^3$, $i=1,2$ are said to belong to the same (non-oriented) \emph{knot equivalence class} if there exists an ambient isotopy $F:\mathbb{R}^3\times [0,1]\rightarrow \mathbb{R}^3$, such that $F(\gamma_1(e^{i \theta}),1)=\gamma_2(e^{i \theta})$ or $F(\gamma_1(e^{i \theta}),1)=\gamma_2(e^{-i \theta})$. 
\end{definition}

\begin{proposition}\label{prop_knot_class}
Given two diffeomorphisms $F_1,F_2:D^2\times S^1\rightarrow T$, where $T$ is a solid torus in $\mathbb{R}^3$,  the knot class of $F_1(1,e^{i\theta})$ is the same as the knot class of $F_2(1,e^{i \theta})$.
\end{proposition}

For the proof of Proposition \ref{prop_knot_class} we need the following lemma.

\begin{lemma}\label{lemma_knot}
Let $\iota: S^1\times S^1 \rightarrow D^2\times S^1$ be the standard inclusion and let $\phi:S^1\times S^1 \rightarrow S^1\times S^1$ be a diffeomorphism. Then there exists a diffeomorphism $\widetilde{\phi}:D^2\times S^1 \rightarrow D^2\times S^1$ extending $\phi$ (i.e $\widetilde{\phi}\circ \iota=\iota\circ\phi$) if and only if the map $p_2(\phi(\cdot,1)):S^1\rightarrow S^1$, where $p_2$ is the projection to the second component, has vanishing degree.
\end{lemma}

\begin{proof}[{Proof of Lemma \ref{lemma_knot}}]
If there is such an extension $\widetilde{\phi}$ then $h_t(e^{ 2\pi i x})=p_2(\widetilde{\phi}(t e^{2\pi i  x},1))$, where $0\leq t\leq 1$, gives a homotopy between $p_2(\phi(\cdot,1))$ and a constant map, and hence $\textrm{deg}(p_2(\phi(\cdot,1)))=0$.

Conversely, any $\phi\in \textrm{Diff}(S^1\times S^1)$ is isotopic to a map $\phi_A$ of the form
\begin{equation}
\phi_A(e^{2\pi i x},e^{ 2\pi i y})=(e^{ 2\pi i (a x+b y)},e^{ 2\pi i (c x+d y)})
\end{equation}
for some $A=\footnotesize{\Big( \begin{array}{cc}
a & b \\
c &  d  \end{array} \Big)}$ in $\textrm{GL}_2(\mathbb{Z})$.
As the degree is a homotopy invariant, for a map $\phi_A$ isotopic to a diffeomorphism $\phi$ with $\textrm{deg}(p_2(\phi(\cdot,1)))=0$ we must have $c=0$. But then
\begin{equation}
\widetilde{\phi}_A(re^{ 2\pi i  x},e^{ 2\pi i y})=(re^{ 2\pi i  (a x+b y)},e^{2\pi i  d y})
\end{equation}
gives an extension of $\phi_A$. Finally, an extension of $\phi$ is obtained by the isotopy extension property.
\end{proof}

\begin{proof}[{Proof of Proposition \ref{prop_knot_class}}]
Consider $\widetilde{\phi}:=F_2^{-1}\circ F_1\in \textrm{Diff}(D^2\times S^1)$. Note that it suffices to prove that $(1,e^{2\pi i s})$ and $\widetilde{\phi}(1,e^{2\pi i s})$ are equivalent knots in $D^2\times S^1$.

Restrict $\widetilde{\phi}$ to a map $\phi\in\textrm{Diff}(S^1\times S^1)$ such that  $\widetilde{\phi}\circ \iota=\iota\circ\phi$. By the proof of the previous lemma we know that $\phi$ is isotopic to a map of the form
\begin{equation}
\phi_A(e^{ 2\pi i  x},e^{ 2\pi i y})=(e^{ 2\pi i  (a x+b y)},e^{2\pi i  d y})\, ,
\end{equation}
where $a,d\in \{\pm 1\}$ and $b \in \mathbb{Z}$. After possibly reversing orientation (which is allowed by our definition of knot equivalence) we can assume that $d=+1$. By the isotopy extension property we thus see that in $D^2\times S^1$ the knots $\widetilde{\phi}(1,e^{2\pi i s})$ and $(e^{ 2\pi i b s},e^{2\pi i s})$ are equivalent.

But in $D^2\times S^1$ the knot $(e^{ 2\pi i b s},e^{2\pi i s})$ is clearly equivalent to the knot $(0,e^{2\pi i s})$, which is in turn equivalent to the knot $(1,e^{2\pi i s})$. This proves the assertion.
\end{proof}

For the sake of convenience, in this paper, instead of closed embedded hypersurfaces $M\subset\mathbb{R}^{n+1}$ we often talk about the compact domain $K$ bounded by $M$. These two points of view are of course equivalent, since $K$ determines $M=\partial K$, and vice versa.

Given a solid torus $T\subset \mathbb{R}^3$ by Proposition \ref{prop_knot_class} we can associate to it a knot class. Given a family of tori $\mathcal{T}$ we denote by $\mathcal{K}(\mathcal{T})$ the set of knot classes. Clearly if $T_1,T_2\in\mathcal{T}$ are isotopic, then they are in the same knot class. The far less obvious part of the statement of Theorem \ref{thm_main}, which we will prove in the bulk of the paper, is that whenever two mean-convex tori $T_1,T_2$ have the same knot class, then there exists a mean-convex isotopy between them. Generalizing mean-convex domains in $\mathbb{R}^3$, we actually consider $2$-convex domains in $\mathbb{R}^{n+1}$ for general $n$, so let us now discuss some basic notions about the domains we consider.

\begin{definition}
A smooth compact domain $K\subset\mathbb{R}^{n+1}$ is called \emph{$2$-convex}, if 
\begin{equation}
\lambda_1+\lambda_2>0
\end{equation}
at all points $p\in\partial K$. Here, $\lambda_1\leq \lambda_2\leq \ldots\leq \lambda_n$ denote the principal curvatures, i.e.~the eigenvalues of the second fundamental form $A$ of $\partial K$.
\end{definition}

\begin{definition}
A \emph{$2$-convex isotopy} is a smooth family $\{K_t\subset\mathbb{R}^{n+1}\}_{t\in [0,1]}$ of $2$-convex domains. Here smoothness is taken with respect to the smooth topology of submanifolds with boundary in $\mathbb{R}^{n+1}$. We say that two $2$-convex domains $K,K'\subset\mathbb{R}^{n+1}$ are \emph{$2$-convex isotopic} if there is a $2$-convex isotopy $\{K_t\}_{t\in [0,1]}$ such that $K_0=K$  and $K_1=K'$.
\end{definition}

\begin{definition}
We say that an isotopy  $\{K_t\}_{t\in [0,1]}$ is \emph{monotone} if $K_{t_2}\subseteq K_{t_1}$ for $t_2\geq t_1$, and \emph{monotone outside a set $X$} if $K_{t_2}\cap (\R^3\setminus X)\subseteq K_{t_1}\cap (\R^3\setminus X)$ for $t_2\geq t_1$. We call such an isotopy $\{K_t\}_{t\in [0,1]}$ \emph{trivial outside a set $X$} if $K_t\cap (\R^3\setminus X)$ is independent of $t$. 
\end{definition}

The goal for the rest of this section is to recall the gluing theorem (and some other required definitions) from \cite{BHH1}. The input for the gluing map is a controlled configuration of $2$-convex domains and curves given by the following three definitions.

We first recall that a $2$-convex domain $K\subset\mathbb{R}^{n+1}$ is called \emph{$\alpha$-noncollapsed} (see \cite{sheng_wang,andrews1,HK}) if each boundary point $p\in \partial K$ admits interior and exterior balls tangent at $p$ of radius at least $\alpha/H(p)$.

\begin{definition}[{$\mathbb{A}$-controlled domains}]\label{def_contr_dom}
Let $\alpha\in (0,n-1)$, $\beta>0$, $c_H>0$, and $C_A<\infty$. A smooth domain $K\subset\mathbb{R}^{n+1}$ is called \emph{$(\alpha,\beta,c_H,C_A)$-controlled}, if it is $\alpha$-noncollapsed and satisfies
\begin{equation}\label{eq_bound_contr_dom}
H\geq c_H ,\qquad \lambda_1+\lambda_2\geq \beta H,\qquad \abs{A} + \abs{\nabla A}\leq C_A.
\end{equation}
We write $\mathbb{A}=(\alpha,\beta,c_H,C_A)$ to keep track of the constants. We denote by $\mathcal{D}$ the set of all (possibly disconnected) $2$-convex smooth compact domains $K\subset \mathbb{R}^{n+1}$, and we denote by $\mathcal{D}_{\mathbb{A}}=\{K\in\mathcal{D}\;|\;\textrm{$K$ is $\mathbb{A}$-controlled} \}$ the subset of all $\mathbb{A}$-controlled domains.
\end{definition}

\begin{definition}[{$b$-controlled curves}]
Let $b>0$. An oriented compact curve $\gamma\subset \mathbb{R}^{n+1}$ (possibly with finitely many components) is called \emph{$b$-controlled} if the following conditions are satisfied.
\begin{enumerate}[\hspace{3mm}(a)]
\item The curvature vector satisfies $|\kappa|\leq b^{-1}$ and $|\partial_s \kappa|\leq b^{-2}$.
\item Each connected component has normal injectivity radius at least $\frac{1}{10}b$.
\item Different connected components are at least distance $10 b$ apart.  
\end{enumerate} 
We denote by $\mathcal{C}_{b}$ the set of all $b$-controlled curves $\gamma\subset\mathbb{R}^{n+1}$.
\end{definition}

\begin{definition}[Controlled configuration of domains and curves]\label{conrolled_conf}
We call a pair $(D,\gamma)\in \mathcal{D}_{\mathbb{A}}\times \mathcal{C}_{b} $ an \emph{$(\mathbb{A},b)$-controlled configuration} if the following holds.
\begin{enumerate}[\hspace{3mm}(a)]
\item The interior of $\gamma$ lies entirely in $\mathbb{R}^{n+1}\setminus D$.
\item The endpoints of $\gamma$ satisfy the following properties:
\begin{itemize}
\item If $p\in \partial\gamma\cap\partial D$, then $\gamma$ touches $\partial D$ orthogonally there.
\item If $p\in \partial\gamma\setminus\partial D$, then $d(p,\partial D)\geq 10b$.
\item $d \big( \gamma \setminus\bigcup_{p\in \partial\gamma}B_{b/10}(p),\partial D \big)\geq b/20$.
\end{itemize}
\end{enumerate}
We denote by $\mathcal{X}_{\mathbb{A},b}\subseteq \mathcal{D}_{\mathbb{A}}\times \mathcal{C}_{b}$ the set of all $(\mathbb{A},b)$-controlled configurations.
\end{definition}

The gluing map from Theorem \ref{thm_attachstrings} below transforms an $(\mathbb{A},b)$-controlled configuration $(D,\gamma)\in \mathcal{X}_{\mathbb{A},b}$ into a $2$-convex domain $K\in \mathcal{D}$. In order to state precisely how $K$ looks like at loose ends of $\gamma$, we also need the following definition of capped-off tubes from \cite[Def.~2.9]{BHH1}. To this end, first recall that given $\mathbb{A}$ we can fix a suitable standard cap $K^{\textrm{st}}=K^{\textrm{st}}(\mathbb{A})$ as in \cite[Def.~2.8]{BHH1} and \cite[Def.~2.2, Prop.~3.10]{HK_surgery}, which is given as domain of revolution of a suitable concave function $u^\textrm{st}:(-\infty,1]\to\mathbb{R}$ and serves as standard piece for the surgeries.
 
\begin{definition}[Capped-off tube]\label{cap_off_def}
Let $\gamma:[a_0,a_1]\rightarrow \mathbb{R}^{n+1}$ be a connected $b$-controlled curve parametrized by arc-length. A \emph{right} \emph{capped-off tube} of radius $r<b/10$ around $\gamma$ at a point $p=\gamma(s_0)$, where $s_0\in [a_0+r,a_1]$, is the domain 
\begin{align*}
CN^+_{r}(\gamma,p)=
\{\gamma(s)+ u^{\textrm{st}}\big(1-\tfrac{s_0-s}{r}\big)\!\cdot\!  \left(\bar{B}_r^{n+1}\cap \gamma'(s)^{\perp}\right) |\;s\in [a_0,a_1]\}.
\end{align*}
A \emph{left capped-off tube} $CN^-_{r}(\gamma,p)$ is defined analogously. If $p$ is an endpoint of the curve $\gamma$, we simply talk about the \emph{capped-off tube} $CN_r(\gamma,p)$ at $p$, meaning $CN^+_{r}(\gamma,p)$ if $p=\gamma(a_1)$ and $CN^-_{r}(\gamma,p)$ if $p=\gamma(a_0)$.\end{definition}

We can now recall the main gluing theorem from \cite{BHH1}.

\begin{theorem}[Gluing map, Theorem 4.1 of \cite{BHH1}]\label{thm_attachstrings}
There exists a constant $\bar{r}=\bar{r}(\mathbb{A},b)>0$, a smooth rigid motion equivariant map
\begin{equation*}
\mathcal{G}:\mathcal{X}_{\mathbb{A},b}\times (0,\bar{r})\rightarrow \mathcal{D},\quad ((D,\gamma),r)\mapsto \mathcal{G}_{r}(D,\gamma),
\end{equation*}
and a smooth increasing function $\delta:(0,\bar{r})\rightarrow \mathbb{R}_+$ with $\lim_{r\rightarrow 0}\delta(r)=0$, such that the following holds.
\begin{enumerate}[\hspace{3mm}(1)]
\item $\mathcal{G}_{r}(D,\gamma)$ deformation retracts to $D\cup \gamma$.
\item We have
\begin{equation*}
\qquad\quad\mathcal{G}_{r}(D,\gamma)\setminus\bigcup_{p\in \partial\gamma} B_{\delta(r)}(p) = D \cup  N_{r}(\gamma)\setminus \bigcup_{p\in \partial \gamma} B_{\delta(r)}(p),
\end{equation*}
where $N_{r}(\gamma)$ denotes the (solid) $r$-tubular neighborhood of $\gamma$. The collection of balls $\{B_{\delta(r)}(p)\}_{p\in\partial \gamma}$ is disjoint.
\item  If $p\in \partial \gamma\setminus\partial D$ and $\gamma_p$ denotes the connected component of $\gamma$ containing $p$ as its endpoint, then
\begin{align*}
\mathcal{G}_{r}(D,\gamma)\cap B_{\delta(r)}(p)= 
CN_{r}(\gamma_p,p)\cap B_{\delta(r)}(p),  
\end{align*}
where $CN_{r}(\gamma_p,p)$ denotes the capped-off $r$-tube around $\gamma_p$ at $p$. 
\item The construction is local: If $(D\cup\gamma)\cap B_{\delta(r)}(p)=(\widetilde{D}\cup\widetilde{\gamma})\cap B_{\delta(r)}(p)$ for some $p\in \partial\gamma$, then $\mathcal{G}_r(D,\gamma)\cap B_{\delta(r)}(p)=\mathcal{G}_r(\widetilde{D},\widetilde{\gamma})\cap B_{\delta(r)}(p)$.
\end{enumerate}
Moreover, in the special case that $\gamma\cap B_{\delta(r)}(p)$ is a straight line and that $D\cap B_{\delta(r)}(p)=\bar{B}_R(q)\cap B_{\delta(r)}(p)$ for some $q\in \mathbb{R}^{n+1}$ and $R>r$, then $\mathcal{G}_r(D,\gamma)\cap B_{\delta(r)}(p)$ is given by the explicit rotationally symmetric construction from \cite[Prop.~4.2]{BHH1}.
\end{theorem}

\section{Transforming $2$-convex domains to marble graphs}\label{sec_marble}

Using the gluing map from Theorem \ref{thm_attachstrings}, we can now define marble graphs by extending our concept of a marble tree from \cite{BHH1}. 

\begin{definition}[Marble graph]\label{def_marblegraph}
A \emph{marble graph} with string radius $r_s$ and marble radius $r_m$ is a domain of the form $G:=\mathcal{G}_{r_s}(D,\gamma)$ such that
\begin{enumerate}
\item $D=\bigcup_i \bar{B}_{r_m}(p_i)$ is a union of finitely many balls (``marbles'') of radius $r_m$.
\item The curve $\gamma$ is such that:
\begin{itemize}
\item there are no loose ends, i.e. $\partial\gamma\setminus\partial D=\emptyset$,
\item $\gamma\cap \bar{B}_{3r_m}(p_i)$ is a union of straight rays for all $i$,
\end{itemize}
\end{enumerate}
If $D\cup \gamma$ is contractible, then $G$ is called a \emph{marble tree}. If $D\cup \gamma$ is homotopy equivalent to $S^1$ then $G$ is called a \emph{marble circuit}. 
\end{definition}

\begin{remark}\label{remark_moreover}
Note that by the moreover-part of Theorem \ref{thm_attachstrings} and by the structure of $(D,\gamma)$ as in Definition \ref{def_marblegraph}, $G$ does not depend on the precise value of the control parameters. We can thus freely adjust the control parameters and shrink $r_m, r_s$ whenever needed. 
\end{remark}

The goal of this section is to prove the following structure theorem for $2$-convex domains.

\begin{theorem}\label{thm_deform_to_marblegraph}
Every $2$-convex domain is $2$-convex isotopic to a marble graph.
\end{theorem}

\begin{proof}
The proof of this theorem uses three steps, generalizing the argument from our previous article \cite{BHH1}. Hence, at some points we only outline the argument and give precise references to steps that have been carried out already in our previous article.

\subsection*{Step 1: Mean curvature flow with surgery.} The theory of mean curvature flow with surgery was first introduced by Huisken-Sinestrari \cite{HuiskenSinestrari} for $2$-convex hypersurfaces $M^n\subset \mathbb{R}^{n+1}$ with $n\geq 3$. Later, a theory for surfaces was developed by Brendle-Huisken \cite{BrendleHuisken} and an approach which works in all dimensions was given by Haslhofer-Kleiner \cite{HK_surgery}. We follow the framework of the last mentioned article here.

Loosely speaking, a mean curvature flow with surgery starting at a $2$-convex domain $K_0$ is a collection of finitely many smooth $2$-convex mean curvature flows $\{K_t^i\subseteq \R^{n+1}\}_{t\in[t_{i-1},t_{i}]}$ (where $i=1,\ldots,\ell$ and $0=t_0<t_1<\ldots< t_\ell$), such that $K^1_{t_0} = K_0$ and $K_{t_i}^+ = K^{i+1}_{t_i}$ is obtained from $K_{t_i}^- = K^i_{t_i}$ by surgeries and/or discarding of connected components. More precisely, one first replaces finitely many (possibly zero) strong $\delta$-necks with center $p_i^j$ and radius $r_{\textrm{neck}}$ by pairs of opposing standard caps $K^{\textrm{st}}$, obtaining the post-surgery domain $K_{t_i}^\sharp$, see \cite[Def.~2.3 and Def.~2.4]{HK_surgery} for a detailed description. Then, one discards finitely many connected components of high curvature to obtain $K_{t_i}^-$. We refer to \cite[Def.~1.17]{HK_surgery} for a precise definition and to \cite[Thm.~1.21]{HK_surgery} for an existence result of such a mean curvature flow with surgery. By comparison with spheres, the flow always becomes extinct in finite time, that is $K_{t_\ell}^+=\emptyset$. In particular, at the last singular time $t_\ell$, there is only discarding of all the remaining components and no surgery.

The existence theorem is accompanied by the canonical neighborhood theorem, which gives a precise description of the regions of high curvature that are discarded. The upshot, see \cite[Cor.~1.25]{HK_surgery} and in particular its proof, is that all discarded components are diffeomorphic to $D^{n+1}$ or $D^{n}\times S^1$. Moreover, the components that are diffeomorphic to $D^{n+1}$ are either (a) convex or (b) capped $\eps$-tubes and the components diffeomorphic to $D^{n}\times S^1$ are (c) $\eps$-tubular loops. The precise definitions for the cases (b) and (c) are as follows.

\begin{definition}[Capped $\eps$-tubes and $\eps$-tubular loops]\label{def_tubes_loops}
\hfill
\begin{enumerate}
\item A \emph{capped $\eps$-tube} is a $2$-convex compact domain $K\subset\mathbb{R}^{n+1}$ diffeomorphic to a ball, together with a controlled connected curve $\gamma\subset K$ with endpoints on $\partial K$ such that:
\begin{enumerate}
\item If $\bar{p}_\pm$ denote the endpoints of $\gamma$ then $K\cap B_{2CH^{-1}(\bar{p}_\pm)}(\bar{p}_\pm)$ is $\eps$-close (after rescaling to unit size) to either (a)  a $(C,\eps)$-cap or (b) a standard-cap $K^{\textrm{st}}$. Here, a $(C,\eps)$-cap is a  strictly convex domain $K\subset\mathbb{R}^{n+1}$ such that every point outside some compact subset of size $C<\infty$ is the center of an $\eps$-neck of radius 1.
\item Every interior point $p\in \gamma$ with $d(p,\bar{p}_+)\geq CH^{-1}(\bar{p}_+)$ and $d(p,\bar{p}_-)\geq CH^{-1}(\bar{p}_-)$  is the center of an $\eps$-neck with axis given by $\partial_s\gamma(p)$. Moreover, if $r$ denotes the radius of the $\eps$-neck with center $p$, then $\gamma$ is  $\eps^{-2}r$-controlled in $B_{\eps^{-1}r}(p)$.
\end{enumerate}
\item An \emph{$\eps$-tubular loop} is a $2$-convex compact domain $K\subset\mathbb{R}^{n+1}$ which deformation retracts to $S^1$, together with a controlled closed curve $\gamma\subset K$ such that every point $p\in \gamma$ is the center of an $\eps$-neck with axis given by $\partial_s\gamma(p)$. Moreover, if $r$ denotes the radius of the $\eps$-neck with center $p$, then $\gamma$ is  $\eps^{-2}r$-controlled in $B_{\eps^{-1}r}(p)$.
\end{enumerate}
\end{definition}

By this theory, we can evolve any $2$-convex domain $K_0$ by mean curvature flow with surgery until it becomes extinct, preserving $2$-convexity and having a precise description of all the discarded components during the process. For a more detailed summary of the essential results described above, we refer to Sections 6 and 8 in our previous article \cite{BHH1}.

\subsection*{Step 2: Isotopies for surgery necks and discarded components.} In \cite{BHH1}, we explained how to geometrically undo the surgery by gluing the two surgery caps back together along a tiny string as described in the gluing construction in Theorem \ref{thm_attachstrings} above.

\begin{lemma}[Combination of Lemma 6.4 and Proposition 6.5 from \cite{BHH1}]\label{lemma_neck_isotopy}
For small enough $\delta$, if $K^\sharp$ is obtained from $K$ by replacing a strong $\delta$-neck (with center $0$ and radius $1$) by a pair of standard caps (with cap separation parameter $\Gamma$), then there is an almost straight line $\gamma$ between the tips of the standard caps and for $r_s$ small enough, there exists an isotopy between $K$ and ${\mathcal G}_{r_s}(K^\sharp,\gamma)$ that preserves $2$-convexity and is trivial outside $B_{6\Gamma}(0)$.
\end{lemma}

Lemma \ref{lemma_neck_isotopy} allowed us to construct an isotopy from a capped $\eps$-tube to a marble graph.

\begin{lemma}[Isotopy for capped $\eps$-tubes, Proposition 7.4  of \cite{BHH1}]\label{lemma_tubes}
For $\eps$ small enough, every capped $\eps$-tube (see Definition \ref{def_tubes_loops}) is $2$-convex isotopic to a marble tree. Moreover, there exists a finite collection $\mathcal{I}$ of $\eps$-neck points with $|p-q|\geq 100\max\{\eps^{-1},\Gamma\}\max\{H^{-1}(p),H^{-1}(q)\}$ for every pair $p,q\in \mathcal{I}$, such that the isotopy is monotone outside $\bigcup_{p\in \mathcal{I}}B_{6\Gamma H^{-1}(p)}(p)$.
\end{lemma}

Similar to this lemma, we now prove a version for $\eps$-tubular loops. This result was not needed for the argument in \cite{BHH1}, as there our topological assumption ruled out discarded components of this type.

\begin{lemma}[Isotopy for $\eps$-tubular loops]\label{lemma_loops}
For $\eps$ small enough, every $\eps$-tubular loop (see Definition \ref{def_tubes_loops}) is $2$-convex isotopic to a marble circuit. Moreover, there exists a finite collection $\mathcal{I}$ of $\eps$-neck points with $|p-q|\geq 100\max\{\eps^{-1},\Gamma\}\max\{H^{-1}(p),H^{-1}(q)\}$ for every pair $p,q\in \mathcal{I}$, such that the isotopy is monotone outside $\bigcup_{p\in \mathcal{I}}B_{6\Gamma H^{-1}(p)}(p)$.
\end{lemma}

\begin{proof}
In the following, we assume that $\eps$ and $r_s$ are small enough. We denote by $K$ an $\eps$-tubular loop as in Definition \ref{def_tubes_loops}. The isotopy from $K$ to a marble circuit is constructed in two steps. First, let $\mathcal{I}\subset \gamma$ be a maximal collection of $\eps$-neck points such that for any pair $p,q\in \mathcal{I}$ the distance between the points is at least $100\max\{\eps^{-1},\Gamma\}\max\{H^{-1}(p),H^{-1}(q)\}$.

For each $p\in\mathcal{I}$, we replace the $\eps$-neck with center $p$ by a pair of standard caps as described in \cite[Def.~2.4]{HK_surgery} and denote the post-surgery domain by $K^\sharp$. This is possible by \cite[Prop.~3.10]{HK_surgery}. Let $\widetilde{\gamma}$ be the disjoint union of almost straight curves connecting the opposing standard caps as in \cite[Lem.~6.4]{BHH1}. Note that $\widetilde{\gamma}$ is Hausdorff close to $\gamma \setminus K^\sharp$. Then, by Lemma \ref{lemma_neck_isotopy}, there exists a $2$-convex isotopy between $K$ and $\mathcal{G}_{r_s}(K^\sharp,\widetilde{\gamma})$.

Note that each connected component of $K^\sharp$ is a small perturbation of a capped-off cylinder (capped-off on both sides) and can therefore be deformed monotonically to a (slightly smaller) capped-off cylinder as described in Definition \ref{cap_off_def}. Then, letting this capped-off cylinder flow by mean curvature, it will instantaneously become strictly convex, so it is certainly $2$-convex isotopic to a round ball.

Let $\{K^{\sharp}_t\}_{t\in[0,1]}$ be the union of the above isotopies between the connected components of $K^{\sharp}$ and balls. Denote by $r_{\min}$ the smallest radius among these balls and let $\{K^{\sharp}_t\}_{t\in[1,2]}$ be an isotopy that concatenates smoothly at $t=1$ and shrinks all balls further to a smaller radius, say $r_{\min}/10$. Moreover, let $\{\widetilde{\gamma}_t\}_{t\in[0,2]}$ be the family of curves which follows $K^{\sharp}_t$ by normal motion starting at $\widetilde{\gamma}_0=\widetilde{\gamma}$. Then $\{\mathcal{G}_{r_s}(K^\sharp_t,\widetilde{\gamma}_t)\}_{t\in [0,2]}$ provides the second part of the $2$-convex isotopy we want to construct, transforming $\mathcal{G}_{r_s}(K^\sharp,\widetilde{\gamma})$ to a marble circuit.
\end{proof}

\begin{remark}
It is trivial to see that also discarded components of type (a) can be transformed to marble graphs. In fact, if $K\subset \mathbb{R}^{n+1}$ is a smooth compact convex domain, we can simply choose a single marble $\bar{B}_{r_m}\subset K$ and then $\{K_t\}_{t\in[0,1]}$ given by $K_t:=t \bar{B}_{r_m}+(1-t)K$ is a monotone convex isotopy (trivial outside $K$), with $K_0=K$ and $K_1=\bar{B}_{r_m}$.
\end{remark}

\subsection*{Step 3: Gluing isotopies together using backwards induction.}  We now fix $\eps$ small enough such that all applications of the previous lemmas as well as the application of \cite[Lem.~9.4]{BHH1} below are justified. Depending on $\eps$ and the parameters of the initial domain $K_0$ we then choose the surgery parameters from \cite{HK_surgery} precise enough, namely the neck precision parameter $\delta$ small enough, and the trigger-curvature, neck-curvature and thick-curvature, as well as their ratios, large enough.

Consider the evolution by mean curvature flow with surgery as described in Step 1, with initial condition $K_0$. Recall that at each $t_i$ finitely many $\delta$-necks with center $p_i^j$ and radius $r_{\textrm{neck}}$ are replaced by a pair of standard caps. Let $B_i^j:=B_{10\Gamma r_\textrm{neck}}(p_i^j)$, and observe that these balls are pairwise disjoint (see \cite[Prop.~2.5]{HK_surgery}). Similarly, for each discarded component $C_i^j$ which is either a capped $\eps$-tube or an $\eps$-tubular loop, by Lemma \ref{lemma_tubes} or Lemma \ref{lemma_loops}, respectively, there is a finite collection of $\eps$-neck points, whose centers and radii we denote by $p_{i}^{jk}$ and $r_{i}^{jk}$. We then set $B_i^{jk}:=B_{10\Gamma r_i^{jk}}(p_i^{jk})$. The isotopy which we will construct will be monotone outside the set of pairwise disjoint balls
\begin{equation}\label{exceptional_set}
X:=\bigcup_{i,j}B_i^j\cup \bigcup_{i,j,k}B_i^{jk}\, .
\end{equation}

Now, let $\A_i$ be the assertion that every connected component of $K^i = K_{t_{i}}^-$ is $2$-convex isotopic to a marble graph. Since at time $t_\ell$ there is only discarding and no surgery, we know that $\A_\ell$ holds, having shown in Step 2 that all discarded components are isotopic to marble graphs. We now prove the following (backwards) inductive step.

\begin{claim}\label{claim_inductionstep}
If $0<i<\ell$ and $\A_{i+1}$ holds, so does $\A_i$.
\end{claim}

\begin{proof}
Smooth evolution by mean curvature flow provides a monotone isotopy between $K_{t_i}^+$ and $K_{t_{i+1}}^-$. Recall that $K_{t_i}^+$ is obtained from $K^i = K_{t_{i}}^-$ by performing surgery on a collection of disjoint $\delta$-necks and/or discarding connected components that are entirely covered by canonical neighborhoods. By the inductive hypothesis, the connected components of $K_{t_{i+1}}^-$ and hence of $K_{t_i}^+$ are isotopic to marble graphs, and by Step 2 the discarded components are isotopic to marble graphs as well. It follows that all components of $K_{t_{i}}^\sharp$ are isotopic to marble graphs. 

Let $\{L_t\}_{t\in [0,1]}$ denote such an isotopy deforming $L_0=K_{t_{i}}^\sharp$ into a union of marble graphs $L_1$, which is monotone outside $X$. We now want to glue together the isotopies of the components. If there was only discarding at time $t_{i}$ there is no need to glue, hence we can assume that $L_0$ has at least two components.

For each surgery neck at time $t_i$, select an almost straight line $\gamma_i^j$ between the tips of the corresponding pair of standard caps and set $\gamma_0=\bigcup_{j}\gamma_i^j$. By Lemma \ref{lemma_neck_isotopy}, the domain $K^i=K_{t_i}^-$ is isotopic to $\mathcal{G}_{r_s}(K_{t_{i}}^\sharp,\gamma_0)$ via $2$-convex domains with an isotopy that is trivial outside $X$. Finally, to get an isotopy $\mathcal{G}_{r_s}(L_t,\gamma_t)$ it remains to construct a suitable family of curves $\{\gamma_t\}_{t\in [0,1]}$ along which we can do the gluing. We start with $\gamma_0$ and then essentially define $\gamma_t$ by following the points where $\gamma_t$ touches $\partial L_t$ via normal motion. It can happen at finitely many times $t$ that $\gamma_{t}$ hits $\partial X$. In this case, we modify $\gamma_t$ according to \cite[Lem.~9.4]{BHH1} to avoid the surgery regions. Then $\mathcal{G}_{r_s}(L_t,\gamma_t)_{t\in [0,1]}$ gives the desired $2$-convex isotopy.
\end{proof}

By backwards induction on $i$, we then in particular obtain that $\A_1$ holds, i.e. $K^1 = K_{t_{1}}^-$ is $2$-convex isotopic to a marble graph. Finally, smooth mean curvature flow provides a $2$-convex isotopy between $K^1$ and $K_0$ (in particular, $K^1$ has only one connected component). We conclude that $K_0$ is $2$-convex isotopic to a marble graph, proving the theorem.
\end{proof}

\section{Deforming marble circuits and conclusion of the proof}\label{sec_circuit}

We will now prove the following deformation theorem for marble circuits.

\begin{theorem}[Marble circuit isotopy]\label{thm_marble_circuit}
Every marble circuit is $2$-convex isotopic to a (solid) thin torus $N_r(\gamma)$, where $\gamma$ is a closed embedded curve and $r$ is arbitrarily small.
\end{theorem}

\begin{proof}
The proof is a generalization of the proof of \cite[Thm.~5.2]{BHH1}, where we showed that every marble tree is $2$-convex isotopic to a round ball. Our previous proof was based on two basics steps: rearrangements and marble reduction. Given some $p\in D\cap \gamma$, the rearrangement step allowed us to push all other curves out of the hemisphere with pole $p$ (see \cite[Lem.~5.3]{BHH1}). The reduction step allowed us (after applying rearrangements) to remove a leaf of the marble tree via $2$-convex isotopies (see \cite[Prop.~5.5]{BHH1}).

By repeatedly choosing a leaf, rearranging and applying reduction, we thus see that every marble circuit is isotopic to a marble circuit $\mathcal{G}_{r_s}(D,\gamma)$ with the property that for every connected component $B$ of $D$, $B\cap \gamma$ consists exactly of two antipodal points.

After possibly shrinking $r_m$ and $r_s$ as in Remark \ref{remark_moreover} and then rescaling the configuration so that $r_m=1$, the moreover-part of Theorem \ref{thm_attachstrings} implies that for each $p\in B\cap \gamma$ the configuration around $p$ is given by the explicit rotationally symmetric gluing model $C_{\varrho^{-1}(r_s)}$ from \cite[Prop.~4.2]{BHH1}. We can then use the isotopy $(C_{\delta})_{\delta\in (\varrho^{-1}(r_s),0.99)}$ around $p$, while increasing the string radius $r_s$ to $\varrho(0.99)$ elsewhere. Finally, using a linear isotopy (see \cite[Prop.~3.12]{BHH1}), this can be perturbed to a solid tube of radius one, which we can deform radially to an arbitrarily thin torus.
\end{proof}

We can now finish the proof of our main theorem.

\begin{proof}[Proof of Theorem \ref{thm_main}]
Let $T,T'\subset \mathbb{R}^{n+1}$ be two $2$-convex tori. If $n=2$, we assume in addition that $T$ and $T'$ have the same knot class. We recall from Section \ref{sec_prelim}, the knot class of a torus is invariant under isotopies.

By Theorem \ref{thm_deform_to_marblegraph} and the topological assumption, $T$ and $T'$ are $2$-convex isotopic to marble circuits $C$ and $C'$, respectively. By Theorem \ref{thm_marble_circuit} the marble circuits $C$ and $C'$ are $2$-convex isotopic to (solid) thin tori $N_r(\gamma)$ and $N_r(\gamma')$, for some closed embedded curves $\gamma, \gamma'$. In the case $n=2$, there exists and ambient isotopy deforming $\gamma$ to $\gamma'$ since we assumed that $T$ and $T'$ have the same knot class. In the case $n\geq 3$, there also exists an ambient isotopy deforming $\gamma$ to $\gamma'$ since there are no nontrivial knots in $\mathbb{R}^{n+1}$ for any $n\geq 3$. Finally, by choosing $r$ very small, it is easy to see that such an ambient isotopy gives rise to a $2$-convex isotopy between $N_r(\gamma)$ and $N_r(\gamma')$.
\end{proof}

\bibliography{BHH-tori}

\bibliographystyle{abbrv}

\vspace{10mm}
Reto Buzano: r.buzano@qmul.ac.uk\\
{\sc School of Mathematical Sciences, Queen Mary University of London, Mile End Road, London E1 4NS, UK}\\

Robert Haslhofer: roberth@math.toronto.edu\\
{\sc Department of Mathematics, University of Toronto, 40 St George Street, Toronto, ON M5S 2E4, Canada}\\

Or Hershkovits: orher@stanford.edu\\
{\sc Department of Mathematics, Stanford University, 450 Serra Mall, Building 380, Stanford, CA 94305-2125, USA}\\

\end{document}